\documentclass{amsart}
\usepackage{amsmath,amssymb}
\usepackage{graphicx,color}
\usepackage[all]{xy}
\newtheorem{introtheorem}{Theorem}

\newtheorem{theorem}{Theorem}[section]
\newtheorem{proposition}[theorem]{Proposition}
\newtheorem{lemma}[theorem]{Lemma}
\newtheorem{corollary}[theorem]{Corollary}

\newtheorem{problem}[theorem]{Problem}

\newtheorem{hyp}[theorem]{Hypothesis}
\newtheorem{definition}[theorem]{Definition}
\newtheorem{example}[theorem]{Example}

\newtheorem{remark}[theorem]{Remark}

\def\E{{\mathcal E}}

\def\Z{{\mathbb Z}}
\def\Q{{\mathbb Q}}
\def\C{{\mathcal C}}

\def\Hom{\mathrm{Hom}}

\def\cat0{\mathrm{cat}_0}

\def\dim{\mathrm{dim}}
\def\ker{\mathrm{ker}}

\def\im{\mathrm{im}}
\def\aut{\mathrm{aut}}

\begin{document}

\title{The  Effect of Cell-Attachment on the   Group of   Self-Equivalences of an $R$-localized Space}

\author{Mahmoud Benkhalifa}
\email{makhalifa@uqu.edu.sa}

\address{Umm Al-Qura University, Mekka, Saudia Arabia}

\author{Samuel Bruce Smith}
\email{smith@sju.edu}

\address{
   Saint Joseph's University,
   Philadelphia, PA 19131}



%

\begin{abstract}   Let $R \subseteq \Q$ be a ring with least non-invertible prime $p$.  Let $X = X^{n} \cup_{\alpha} (\bigcup_{j \in J}  e^{q})$ be a cell attachment with $J$ finite and $q$ small with respect to $p.$      Let $\E(X_R)$ denote the group  of  homotopy    self-equivalences    of the   $R$-localization $X_R$.  We use DG Lie models    to construct a short exact sequence  $$ 0 \to \bigoplus_{j \in J}\pi_q(X^n)_R \to  \E(X_R) \to \C^q \to 0$$  where  $\C^q$   is a subgroup of $\mathrm{GL}_{|J|}(R) \times \E(X^n_R)$.      We obtain a related result for the $R$-localization of the nilpotent group $\E_*(X)$ of classes inducing the identity on homology. We deduce  some explicit calculations of both groups   for  spaces with few cells.         \end{abstract}
 \maketitle
\section{Introduction}
 Let $X$ be a finite, simply connected  CW complex.   Let $\E(X)$ denote  the  group of homotopy equivalence classes of   homotopy self-equivalences of $X$.    Let
 $\E_*(X)$ denote the  subgroup represented by self-equivalences that induce the identity map on $H_*(X; \Z)$.    
 The study of the groups $\E(X)$ and $\E_*(X)$ by means of a cellular decomposition of $X$ is  a difficult problem with   a long history.    See Rutter \cite[Chapter 11]{Ru} for a  survey.

In \cite{DZ},    Dror-Zabrodsky proved   $\E_*(X)$ is a nilpotent group.  Maruyama \cite{Mar}  then proved $\E_*(X)_R \cong \E_*(X_R).$ These results together
 opened the door to  the use of algebraic models  for studying the localization of  nilpotent self-equivalence groups. The group  $\E_*(X)$ was studied in \cite{A-M}.   The rationalization of   the subgroup $\E_\sharp(X)$ of self-equivalences inducing the identity on homotopy groups has been studied extensively using Sullivan models (c.f.   \cite{AL1, A-S, FF}). 
  
  The  group $\E(X_\Q)$ has  emerged as a recent object of interest.  Arkowitz-Lupton \cite{AL} gave the first examples of  finite groups occurring as  $\E(X_\Q)$.  Further examples    were given by the first named author in \cite{Benk}.     Costoya-Viruel \cite{CV}  then proved  the remarkable result that every finite group $G$ occurs as    $ G \cong \E(X_\Q)$ for some finite $X$. Again, all this work was also accomplished using Sullivan models.     
The purpose of this paper is to explore the use of  Anick's and Quillen's DG Lie algebra models for     studying the groups  $\E(X_R)$ and $\E_*(X)_R.$

 We briefly recall the main result of Anick's  and Quillen's theories now in order to establish our overriding hypotheses. 
Let $R \subseteq \Q$  be a subring  with least non-invertible prime $p \in R$.   When $R=\Q$ set  $p = +\infty.$  Let $\textbf{CW}_{m}^{k+1}$ denote the category  of $m$-connected, finite
 CW complexes of dimension no greater than $k+1$ with $m$-skeleton reduced to a point. Let $\textbf{CW}_{m}^{k+1}(R)$  denote the  category obtained by $R$-localizing the spaces in
$\textbf{CW}_{m}^{k+1}.$
By   Anick \cite{An1,An2},  when $k < \text{min}(m+2p-3,mp-1)$ the homotopy category of $\textbf{CW}_{m}^{k+1}(R)$
is   equivalent to  the homotopy category of $\textbf{DGL}_{m}^{k}(R)$ consisting of  free differential graded (DG) Lie algebras  $(L(V),\partial)$   in which $V$ is a free $R$-module  satisfying $ V_{n} = 0$ for  $
n< m$ and $n > k$.  When $R = \Q$ the corresponding result for $m = 1$ and any $k$ is due to Quillen \cite{Q}. 
Summarizing, we have:
\begin{hyp} \label{eq:hyp}  {\em We assume that $R \subseteq \Q$ is a ring with least non-invertible prime $p$.  With $R$ fixed,  we take $1 \leq m < k$ satisfying $k < \text{min}(m+2p-3,mp-1)$.  By a space $X$ 
we always mean an object in $\mathbf{CW}^{k+1}_{m}$, an $m$-connected finite CW complex with top degree cells of dimension  $\leq k+1$ 
When $R= \Q$ we assume $m=1$ and $k$ is finite.}  
\end{hyp} 
Let  $X$ be an object in $\mathbf{CW}^{k+1}_{m}$.   Write $X^n$ for the $n$-skeleton of $X$.  We consider the situation in which   $$X = X^q = X^n \cup_\alpha \left( \bigcup_{i=1}^j e_i^q \right)$$   is  the space obtained by attaching  $q$-cells to  a space $X^n$ for $ n < q \leq k+1$ by a map $\alpha \colon \bigvee_{i} S^{q-1} \to X^n.$

\begin{introtheorem} \label{main}  Given  $R \subseteq \Q$  and $X$   satisfying Hypothesis \ref{eq:hyp}, there are  short exact sequences 
\begin{equation} 0 \to \bigoplus_{i=1}^{j}\pi_{q}(X^{n}_R)  \to \mathcal{E}(X_R)
 \to \mathcal{C}^{q} \to 0 \end{equation}  
\begin{equation}  0 \to  \bigoplus_{i=1}^{j} \pi_{q}(X^{n}_R)  \to \mathcal{E}_*(X_R)
 \to \mathcal{C}^q_{*} \to 0 \end{equation} 
 Here   $\mathcal{C}^{q} \subseteq \mathrm{GL}_j(R) \times \E(X^{n}_R)$.  The subgroup
 $\mathcal{C}_*^{q} \subseteq  \C^q$ is contained in $\mathrm{GL}_k(R) \times \E_*(X^{n}_R)$ with
$k$   the rank of   the linking homomorphism $\pi_{n+1}(X^{q}, X^{n})_R\to \pi_{n}(X^{n}, X^{n-1})_R$  in the long exact sequence of the triple.  In particular, $\mathcal{C}_*^{q} \subseteq  \E_*(X^{n}_R)$ when $q > n+1$.
\end{introtheorem}

We prove Theorem \ref{main} in Section \ref{sec:Wh}. In Section \ref{sec:R-local}, we deduce some consequences for  $R$-localized spaces. 

Given a finitely generated abelian group $G$,  let  $M(G, m)$ denote the Moore space.    
 Barcus-Barrett   \cite{B-B} proved the homology representation $\E(M(G, m)) \to \aut(G)$ is surjective and identified the kernel as an Ext-group.   
It is a classical open problem to complete this calculation  \cite[Problem 8]{Kahn} with many partial results and extensions (cf. Rutter [op. cit] and Baues \cite{Ba}). We obtain the following result in Anick's category of $R$-local spaces:

\begin{example}
\label{exmain} 
Let $G_1, \ldots, G_r$  be finitely generated abelian groups and $R \subseteq \Q$.    Let $$X=M(G_1,m+1)\vee M(G_2,m+3) \vee \cdots \vee M(G_r, m +2r-1)$$    with $r \leq m/2$ and $m, k= m+2r$ and $ R$ satisfying Hypothesis \ref{eq:hyp}.  Then 
$$
    \E(X_R)\cong \prod_{i=1}^{r} \aut((G_i)_R).
$$ 
\end{example}
We obtain a related calculation for the group $\E_*(X)_R$: 
 \begin{example} \label{exmain2}     Let $G$ be a finitely generated abelian group and $R \subseteq \Q$  
Let $X$    in $\textbf{CW}_{m}^{k+1}$  and $R$ be as in Hypothesis \ref{eq:hyp} with $X$ having a  cellular decomposition
$$M(G,m+1)=X^{m+2}\subset X^{m+3}\subset\cdots\subset X^{2m}=X.$$
Assume that the linking homomorphisms $\pi_{r+1}(X^{r+1}, X^{r})_R\to \pi_{r}(X^{r}, X^{r-1})_R$ vanish for $r = m+2, \ldots, 2m-1$.  Let $j = \dim_R(H_{2m}(X, X^{2m-1};  R))$. 
  Then
$$ \E_*(X)_R \cong \bigoplus_{i=1}^{j} [G_R, G_R].$$
 Here $[G_R, G_R]$ is additive group corresponding to the sub $R$-module  of $G_R \otimes G_R$ generated by elements of the form $x \otimes y - (-1)^m y \otimes x.$   
\end{example}
We give  a complete calculation of both groups in a simple case.  Let $R^* = \aut(R)$ denote the group of units of $R$. 
\begin{example}  
\label{exmain3}  
Let $m, m+n+1$ and $R \subseteq \Q$   satisfy Hypothesis \ref{eq:hyp}.   Then
$$\begin{array}{ll} \E \left( (S^{m+1} \times S^{n+1})_R \right) & \cong \left\{  \begin{array}{ll}  R^*  \times  R^*  & \hbox{for }  n \neq 2m \hbox{ or } m \hbox{ even.} \\
 R \oplus (R^*  \times  R^*) &  \hbox{for } n = 2m \hbox{ and } m \hbox{ odd }
\end{array} \right.
\\ \\
 \E_* (S^{m+1} \times S^{n+1})_R  & \cong \left\{  \begin{array}{ll}   0 & \hbox{for }  n \neq 2m \hbox{ or } m \hbox{ even.}  \\
R  & \hbox{for } n = 2m \hbox{ and } m \hbox{ odd }.
\end{array} \right.
\end{array}
$$  
\end{example}
We prove  two results concerning the  question of finiteness   $\E(X_R)$.  First:    \begin{theorem}
\label{finite1}  Let $R \subseteq \Q$ and   $X$ in  $\textbf{CW}_{m}^{k+1}$ be as in Hypothesis \ref{eq:hyp}.  Suppose  $\pi_{n+1}(X^{n+1}, X^{n})_R \neq 0$ and    the linking homomorphism   $\pi_{n+1}(X^{n+1}, X^{n})_R \to \pi_n(X^n, X^{n-1})_R$ vanishes.  Then  $\E(X^n_R)$   finite implies  $\E(X^{n+1}_R)$ is infinite.
\end{theorem}
Finally, we  give a calculation to indicate that  finiteness  of $\E(X_R)$ requires a reasonably large CW complex $X$.   \begin{theorem}    Let $X$ be a simply-connected finite CW complex  of dimension $\leq 5.$  Let $ R \subseteq \Q$ have least invertible prime $p \geq 7.$  If $p$ is finite, assume the linking homomorphisms $\pi_{r+1}(X^{r+1}, X^{r})_R\to \pi_{r}(X^{r}, X^{r-1})_R$ vanish for $r =  2,3,  4$ vanish.     Then $\E(X_R)$ is infinite. \end{theorem}

  \section{Homotopy Self-Equivalences of DG Lie Algebras} \label{sec:Wh}
Let $(L(V),\partial )$ be an object in $\textbf{DGL}_{m}^{k}(R)$.  Recall this means  $V$ is  a  free graded $R$-module concentrated in degrees $n$ with $m \leq n \leq  k$    and $L(V)$ is the free graded Lie algebra over $R$.  Write $V_{<n} = \bigoplus_{i = m}^{n-1}   V_i.$ The differential $\partial$ is of degree $-1$.  For each $n \leq k$, $\partial$    induces a differential   $\partial_{< n}$ on  $L(V_{< n})$ making $(L(V_{<n}), \partial_{<n})$ a sub DG Lie algebra.   We write the  homology  as $H_*(L(V_{< n}))$, suppressing the differential. 
 The linear part of  $\partial$  gives a differential $d$ on  $V$.    
The  homology  $H_*(V) = H_*(V, d)$    can be identified with the graded module of
indecomposable generators  of $(L(V),\partial)$.

Homotopies between maps $\alpha, \alpha' \colon (L(V), \partial) \to (L(W), \delta)$ in $\textbf{DGL}_{m}^{k}(R)$ are defined by means of the Tanr\'{e} cylinder  (cf. \cite[Ch.II.5]{Tanre} and \cite[p.425-6]{An1}).   
Let  $$(L(V), \partial)_I = (L(V, sV, V'),D)$$  be the DG Lie algebra with $V\cong V'$ and $(sV)_{i}=V_{i-1}$.  Let  $S$ denote the derivation of degree +1 on $L(V, sV, V')$ with $S(v)=sv$ and $S(sv)=S(v')=0$. The  differential $D$ is given by
$D(v)=\partial(v), D(sv)=v'$ and $ D(v')=0$.  The degree zero derivation $\theta=D\circ S+S\circ D$ of $L(V), \partial)_I$ gives rise to an automorphism $e^\theta$ of    $L(V, \partial)_I.$ We note that for $v \in V$  
  $$e^{\theta}(v)=v+v'+\underset{n\geq 1}{\sum} \frac{1}{n!}(S\circ D)^{n}(v). $$
Define  $\alpha \simeq  \alpha'$  if there is  a DG Lie morphism  
 $$F \colon   (L(V, sV, V'),D)\to (L(V),\partial)$$
 satisfying $F(v)=\alpha(v)$ and $F\circ e^{\theta}(v)=\alpha'(v).$  

Quillen \cite{Q} and Anick \cite{An1} proved  that, under Hypothesis \ref{eq:hyp}, there is an assignment $X \mapsto (L(V), \partial)$  setting up an    equivalence between the homotopy categories of  $\textbf{CW}_m^{k+1}(R)$ and $\textbf{DGL}_{m}^{k}(R)$.   The model $(L(V), \partial)$ recovers  $R$-local homotopy invariants of $X$ via   isomorphisms  (with shifts)
$$ \widetilde{H}_*(X; R) \cong H_{*-1}(V) \hbox{\, and \, } \pi_*(X)_R \cong H_{*-1}(L(V)).$$  
As for self-equivalence groups, their results directly imply identifications:    $$\E(X_R) \cong \frac{\aut(L(V), \partial))}{\simeq} \hbox{\, and \, } \E_*(X_R) \cong \frac{\aut_*(L(V), \partial))}{\simeq}$$ 
Here $\aut(L(V), \partial)$ is the group    of DG Lie homotopy self-equivalences  and the subgroup  $\aut_*(L(V), \partial)$ consists of  maps inducing the identity   automorphism of $H_*(V).$
  We  write $$\E(L(V), \partial) = \aut(L(V), \partial)/\simeq \hbox{\, and \, } \E_*(L(V), \partial) = \aut_*(L(V), \partial)/\simeq$$ for these algebraic equivalence groups.

 Now consider a cellular attachment  $X = X^{n} \cup_{\alpha} (\bigcup_{i =1}^{j} e_i^{q})$ as in Theorem \ref{main}.   Let $(L(V), \partial)$ denote the DG Lie algebra model for $X$.  Then $V = V_{q-1} \oplus V_{< n}$ where $\dim V_{q-1} = j,$  the number of $q$-cells attached. A homotopy self-equivalence $f \colon X \to X$ induces a DG Lie algebra isomorphism $\alpha \colon (L(V), \partial) \to (L(V), \partial).$ Let $\widetilde{\alpha}_{q-1} \colon V_{q-1} \to V_{q-1}$ denote the map induced on $V_{q-1}$ by   restricting $\alpha$  and then projecting to $V_{q-1}$.   Let  $\alpha_{< n} \colon (L(V_{< n}, \partial_{< n}) \to (L(V_{< n}), \partial_{< n})$ denote the DG Lie   algebra restriction map.  We then obtain a commutative square of the form: 
$$\xymatrix{ V_{q-1} \ar[rr]^{\widetilde{\alpha}_{q-1}} \ar[d]^{B_{q-1}} &&  V_{q-1} \ar[d]^{B_{q-1}} \\
H_{q-2}(L(V_{< n})) \ar[rr]^{H(\alpha_{< n})} && H_{q-2}(L(V_{< n}))}
$$
Here  $B_{q-1} \colon V_{q-1} \to H_{q-2}(L(V_{<  n}))$  is given by $$B_{q-1}(v_{q-1})= \{ \partial_{q-1}(v_{q-1}) \} \in H_{q-2}(L(V_{<  n})).
$$
where $\left\{  \ \  \right\}$    denotes the homology class of a cycle. 
We use this diagram to define the group $\C^q$.   Given a self-equivalence $\alpha$ of $L(V)$ we write $[ \alpha]$ for the homotopy equivalence class in $\E(L(V), \partial).$  
\begin{definition}
\label{7}
 Let  $\C^{q}$ the subset of pairs  $(\xi, [\gamma]) \in \aut(V_{q-1}) \times \E(L(V_{< n}), \partial_{<n})$ such that   the following  diagram is commutative:
$$\xymatrix{ V_{q-1} \ar[rr]^{\xi} \ar[d]^{B_{q-1}} &&  V_{q-1} \ar[d]^{B_{q-1}} \\
H_{q-2}(L(V_{< n})) \ar[rr]^{H(\gamma)} && H_{q-2}(L(V_{< n}))}
$$
\end{definition}
\begin{proposition}
\label{p2}
 $\C^{q}$ is a subgroup of $\aut(V_{q-1}) \times \E(L(V_{< n}), \partial_{< n})$
\end{proposition}
\begin{proof}
Straightforward.
\end{proof}
Define $\Psi_{q}\colon   \E(L(V_{q-1}\oplus V_{<
n}), \partial) \to \C^q$ by setting:
$$
    \Psi_{q}([\alpha])=(\tilde{\alpha}_{q-1},[\alpha_{< n}])
$$
\begin{proposition}
\label{p3}
The map  $\Psi_{q}$   is a surjective homomorphism
\end{proposition}
\begin{proof}
It is easy to see $\Psi_q$ is a homomorphism.  We prove surjectivity. Let $(\xi, [\gamma]) \in\C^{q}$. Choose $\{v_{i}\}_{i\in J}$ as a basis of $V_{q-1}$.  From the hypothesized commutative diagram we obtain   
$(\gamma \circ\partial-\partial\circ\xi)(v_{i}) \in \im
\,\partial_{< n}.$    Choose $u_{i}\in
L( V_{< n})$ of degree $q-1$ with
$
(\gamma \circ\partial-\partial\circ\xi)(v_{i})=\partial_{<
n}(u_{i}).
$
Define $\alpha \colon   (L( V_{q-1}\oplus V_{< n} )
,\partial)\rightarrow(L( V_{q-1}\oplus V_{< n}),\partial )$  by setting
$
\alpha(v_{i})=\xi(v_{i})+u_{i}$  for 
$v_{i}\in V_{q-1}$ and $\alpha=\gamma \text{ on }V_{< n}.
$ and then extending. Then $\alpha$ is clearly an automorphism of $L(V)$.  Observe   \begin{eqnarray*}
\partial\circ\alpha(v_{i})=\partial(\xi(v_{i}))+\partial_{< n}(u_{i})=\gamma \circ\partial(v_{i})= 
\alpha\circ\partial(v_{i})
\end{eqnarray*}
Thus  $\alpha$ represents a class in  $\E(L(V), \partial))$  satisfying  $\Psi_q([\alpha]) = (\xi, [\gamma]).$
\end{proof}

We next identify  
$$\ker\Psi_q =\Big\{[\alpha]\in \mathcal{E}(L(
V), \partial) \mid \widetilde{\alpha}_{q-1}=\mathrm{id}_{V_{q-1}}, \,\,\,\, \alpha_{< n} \simeq \mathrm{id}_{L(V_{< n})} \Big\}.
$$
Let $[\alpha] \in \ker \,  \Psi_q. $ That  $\widetilde{\alpha}_{q-1} = \mathrm{id}_{V_{q-1}}$ means  for all $v \in V_q$ we have $\alpha(v) - v \in L^{q-1}(V_{< n}).$ Here $L^{q-1}(V_{<n})$ denotes the space of elements of $L(V_{<n})$ of degree $q-1.$ Define  
$$\varphi_\alpha \colon V_{q-1} \to L^{q-1}(V_{< n}) \hbox{\, by \, } \varphi_\alpha(v)= \alpha(v)-  v \in L^{q-1}(V_{< n}) \hbox{\, for \,} v \in V_{q-1}.$$  
We next prove that the class $[\alpha]$ has a representative $\beta$ such that  $\varphi_\beta(V)$ 
is contained in the cycles of $L^{q-1}(V_{<n})$:  \begin{lemma}  \label{lem1} Let $[\alpha] \in \ker \,  \Psi_q$.  Then there exists $[\beta] \in \ker \, \Psi_q$   satisfying 
 \begin{itemize}  
\item[{\em (i)}] $\partial(\varphi_\beta(v)) = 0$ for all $v \in V_{q-1}$   
 \item[{\em (ii)}]  $\beta_{< n} = \mathrm{id}_{L(V_{<
   n})}.$   
   \item[{\em (iii)}]  $\alpha \simeq \beta.$
 \end{itemize}
 \end{lemma}
 \begin{proof} Since $\alpha_{< n} \simeq \mathrm{id}_{L(V_{<
   n})}$ there is a homotopy $F \colon (L(V_{<n}), \partial_{< n})_I \to (L(V_{<n}), \partial_{< n})$
   satisfying $F(v) = v$ and $F \circ e^\theta(v) = \alpha_{< n}$ for $v \in V_{< n}$.  Define $\beta$ by setting 
   $$ \beta(v) = \left\{ \begin{array}{ll} v & \hbox{for \, } v \in V_{<n} \\
   \alpha(v) - F\left(\sum_{n \geq 1} \frac{1}{n!} (S \circ D)^n(v) \right) & \hbox{for \, } v \in V_{q-1}. 
   \end{array} \right.
   $$
  Given $v \in V_{q-1}$ we compute: 
   $$ \begin{array}{lll} 
\partial(\varphi_\beta(v)) &=& \partial\left(\alpha(v) - F\left(\sum_{n \geq 1} \frac{1}{n!} (S \circ D)^n(v) \right) - v\right) \\  
&=& \alpha_{<n}(\partial( v)) - \partial \circ F\left(\sum_{n \geq 1} \frac{1}{n!} (S \circ D)^n(v) \right)  
- \partial(v) \\  
& = & F  \circ e^{\theta}(\partial (v)) -    F \circ D \left(\sum_{n \geq 1} \frac{1}{n!} (S \circ D)^n(v) \right)  
- \partial(v) \\  
& = & F \circ D \circ e^{\theta}(v) -    F \circ D \left(\sum_{n \geq 1} \frac{1}{n!} (S \circ D)^n(v) \right)  
- \partial(v) \\  
& =& F\circ D(v+v') - \partial(v) \\  & = & 0.
\end{array} 
 $$
 Thus $\beta$ satisfies (i).  For (ii), we define $G \colon (L(V), \partial)_I \to (L(V), \partial)$ by setting $G = F$ on $(L(V_{< n}), \partial_{<n})_I$ while, for $v \in V_{q-1},$ we set
 $G(v) = \beta(v) \hbox{\, and \, }  G(v')= G(sv) = 0.$
 It is easy to check that $G$ is a DG Lie algebra map.  Given $v \in V_{q-1}$, we have
 $$\begin{array}{lll}
 G \circ e^\theta(v) & = &  G\left(v + v' + \sum_{n \geq 1}\frac{1}{n!} (S \circ D)^n(v)\right) \\  
 & = & G(v) + G\left(  \sum_{n \geq 1}\frac{1}{n!} (S \circ D)^n(v)\right) \\  
 & = & \beta(v) + F\left(  \sum_{n \geq 1}\frac{1}{n!} (S \circ D)^n(v)\right) \\  
 & = & \alpha(v).
 \end{array}$$
  \end{proof} 

Using Lemma \ref{lem1} (i), we define a map 
$$\Theta_q \colon   \ker \Psi_{q}\to
\Hom(V_{q-1}, H_{q-1}(L(V_{< n})))\hbox{\, by \, }   \Theta_q([\beta])(v)=\{\varphi_\beta(v)\} \text{\, for \,} v \in V_{q-1}$$ where $\beta$ is chosen  as in Lemma \ref{lem1}.  \begin{proposition}
\label{p1} The map $$\Theta_q \colon   \ker \Psi_{q}\to
\Hom(V_{q-1}, H_{q-1}(L(V_{< n})))$$  is an isomorphism.
\end{proposition}
\begin{proof} 
First we prove that $\Theta_q$ is well-defined. Suppose
$\beta \simeq \beta'$ satisfy the conclusion of Lemma \ref{lem1}.   Since both maps then restrict to the identity on $L(V_{< n})$, the     homotopy $F \colon (L(V), \partial)_I \to (L(V), \partial)$ between them can be chosen so that $F(V'_{<n}) = F(sV_{<n}) =0$.  
Given $v\in V_{ q-1}$ suppose $\varphi_\beta(v) = \{y \}$ and $\varphi_{\beta'}(v) = \{ y' \}$ for cycles $y, y' \in L^{q-1}(V_{<n}).$   We then have
$$\begin{array} {lll}
y'-y &=& \beta'(v)-\beta(v) \\ & = & F\circ e^{\theta}(v)-F(v)\\    & = &F(v)+F(v')+F\Big(\underset{n\geq 1}{\sum} \frac{1}{n!}(S\circ\partial)^{n}(v)\Big)-F(v) \\  
&=&   F(D(sv))+F\Big(\underset{n\geq 1}{\sum} \frac{1}{n!}(S\circ\partial)^{n}(v)\Big)  \\  
& = &    \partial(F(sv))  \end{array}$$
Thus $y'-y$ is a boundary.  

 It is easy to check $\Theta_q$ is a homomorphism. For injectivity, suppose  $\Theta_q([\beta])(v) =\Theta_q([\beta'])(v)$ in $H_{n+1}(L(V_{\leq n-1}), \partial_{<n})$ for all $v \in V_{q-1}.$ Then $\im\{ \beta - \beta' \colon L(V) \to L(V)\}$ is contained in an acyclic sub DG Lie algebra of $(L(V), \partial)$. Thus   $\beta \simeq \beta'$ by \cite[Prop.II.5(4)]{Tanre}.  

Finally, given a homomorphism  $\psi \in \Hom(V_{q-1}, H_{q-1}(L( V_{<n}), \partial_{<n}))$,  we define
$\beta\colon   (L(V),\partial)\to (L(V),\partial)$ by:
$$\beta(v)=v+\psi(v) \,\text{ for } v \in V_{q-1}\,\,\,\,\,\,\,\,\,\text{ and }\,\,\,\,\,\,\,\,\, \beta =\mathrm{id} \,\text{ on } V_{< n}.$$
Then $\beta$ is a DG Lie morphism with $\Theta_q([\beta])=\psi$. \end{proof}

Summarizing,  we have proven:

\begin{theorem}\label{t1}
   Let $(L(V),\partial)$ be an object in $\textbf{DGL}_{m}^{k}(R)$  with $V = V _{q-1} \oplus V_{< n}$ for $q > n.$  Then   there exists
   a short exact sequence of groups:
$$
0 \to\Hom(V_{q-1},   H_{q-1}(L(V_{< n}))) \to
\mathcal{E}(L(V), \partial) 
\to \C^{q} \to 0.  \qed
$$

\end{theorem}

The first  exact sequence in Theorem \ref{main} is a direct consequence. 
\begin{proof}[Proof of Theorem \ref{main} Part (1)]
The result  follows from  Theorem \ref{t1},  the   isomorphisms $\E(X_R) \cong \E(L(V), \partial), $   $\pi_{q}(X^{n})_R \cong  H_{q-1}(L(V_{< n}) )$
and the identification of $V_{q-1}$ with the  free $R$-module  with generators corresponding to the  $q$-cells of $X$.  See \cite[Theorem 8.5]{An1}.
\end{proof}
We now focus on the group $\E_*(L(V), \partial)$ and the proof of Theorem \ref{main} Part (2). 
Again, we take $(L(V), \partial)$ to be an object in $\textbf{DGL}_m^k(R )$ with  $V = V_{q-1} \oplus V_{< n}$ for $m \leq n < q \leq k.$  When $q =n+1$ we must take into account the 
linear differential $d_{n} \colon V_{n} \to V_{n-1}.$  Since $V_n$ is a free $R$-module
we may choose a subspace $W_n$ of $V_n$ complementary to the $d_n$-cycles in $V_n$   giving $V_n = (\ker \, d_n) \oplus W_n.$  When $q > n+1$ we set $W_{q-1} = V_{q-1}.$

Let $\alpha \in \aut_*(L(V), \partial)$ be given and, as usual, let $\widetilde{\alpha}_{q-1} \colon V_{q-1} \to V_{q-1}$ denote the induced map.  Since $\alpha$ induces  the identity on $H_*(V)$, we see $\widetilde{\alpha}_{q-1}$ fixes $\ker \, d_{q-1} = H_{q-1}(V)$. It follows that $\alpha$   induces a map $\widetilde{\alpha}_{q-1}' \colon W_{q-1} \to W_{q-1}$.  
\begin{definition}
\label{C*}
 Let  $\C_*^{q}$ denote the subset of pairs  $(\chi, [\eta]) \in \aut(W_{q-1}) \times \E_*(L(V_{< n}), \partial_{<n})$ such that   the following  diagram is commutative:
$$\xymatrix{ W_{q-1} \ar[rr]^{\chi} \ar[d]^{B_{q-1}} &&  W_{q-1} \ar[d]^{B_{q-1}} \\
H_{q-2}(L(V_{< n})) \ar[rr]^{H(\eta)} && H_{q-2}(L(V_{< n}))}
$$
\end{definition}
\begin{remark}  If $q > n+1$ or if $d_n = 0$ then $W_n = \{ 0 \}.$  In this case $$C^q_* = C^q \cap \E_*(L(V_{< n}), \partial_{<n}).$$
\end{remark}
We prove: 
\begin{theorem}
\label{t5}
Let  $(L(V),\partial)$ be an object in $\textbf{DGL}_{m}^{k}(R)$ with $V = V _{q-1} \oplus V_{< n}$ for $q > n.$    Then there is a short exact sequence:
$$
0 \to\Hom(V_{q-1},   H_{q-1}(L(V_{< n}))) \to
\mathcal{E}_*(L(V), \partial) 
\to \C_*^{q} \to 0
$$
\end{theorem}

\begin{proof}
Define $\Gamma_q \colon \mathcal{E}_*(L(V), \partial) \to \C^q_*$ by $\Gamma_q(\alpha) = (\widetilde{\alpha}_{q-1}', [\alpha_{<n}]).$  We claim $\Gamma_q$ is surjective.  For given $(\chi, [\eta]) \in C_*^q$ as in Definition \ref{C*}, we can extend $\chi$ to a map $\xi \colon V_{q-1} \to V_{q-1}$ by setting $\xi = \mathrm{id}$ on $\ker \, d_{q-1}.$  Then the pair $(\xi, [\eta]) \in C^q$ and so, by Proposition \ref{p3}, there exists $[\alpha] \in \E(L(V), \partial)$ with $\Psi([\alpha]) = (\xi, [\eta])$ and, further, $\alpha$ may be chosen in $\aut_*(L(V), \partial)$ since $\xi$ and $\eta$ fix $H_*(V).$ Thus $\Gamma_q([\alpha]) = (\chi, [\eta]).$ Finally, observe that $\ker \, \Gamma_q = \ker \, \Psi_q$ and the result follows from Proposition \ref{p1}.
\end{proof}
We deduce: 
\begin{proof}[Proof of Theorem \ref{main} Part (2)] The result follows again from  the Quillen-Anick identifications as in the proof of Theorem \ref{main}, above.  In this case, we note that the  linear differential $d_{n} \colon V_n \to V_{n-1}$ corresponds   linking homomorphism $\pi_{n+1}(X^{n+1}, X^{n})_R\to \pi_{n}(X^{n}, X^{n-1})_R$  in the long exact sequence of the triple.  
\end{proof}
\begin{remark} The exact sequences in Theorems \ref{t1} and \ref{t5} do not split in general.  One simple criterion for splitting occurs when $\partial_{< n}  =0 $. 
\end{remark}

\section{Self-Equivalences of $R$-Local Spaces} \label{sec:R-local}
We begin with a result on the full group $\E(X_R)$.  The following   was stated in the introduction as  Example \ref{exmain}.
\begin{theorem}
\label{thm:exmain} Let $G_1, \ldots, G_r$  be finitely generated abelian groups.  Given $R \subseteq \Q,$
  suppose $X = M(G_1,m+1) \vee M(G_2,m+3) \vee \cdots \vee M(G_r, m +2r-1)$  for $r \leq  m/2$ and $X$ an object of  $\textbf{CW}_{m}^{k+1}$ satisfying Hypothesis \ref{eq:hyp}.  Then
$$
    \E(X_R)\cong \prod_{i=1}^{r}\aut((G_i)_R).
$$
\end{theorem}
\begin{proof}
Recall $M(G,i)$ is $(i-1)$-connected and of dimension $ \leq i+1$.    It follows that, in  the Anick model  $(L(V), \partial)$ for $X$, we have     $V=V_{m}\oplus V_{m+1} \oplus \cdots \oplus V_{m+2r-1}$ with  $\partial = d$   purely linear and taking the form   $$\xymatrix{  V_{m+2r-1}\ar[rr]^{d_{m+2r-1}} && V_{m+2r-2}\ar[r]^{0} & V_{m+ 2r -3}\ar[rr]^{d_{m+2r-3}}&&V_{m+2r-4 }\ar[r]^{0} & \cdots \\
\cdots \ar[r]^0 & V_{m+3} \ar[r]^{d_{m+3}} & V_{m+2} \ar[r]^0 & V_{m+1} \ar[r]^{d_m} & V_m \ar[r] & 0.}$$
Here  $$(G_i)_R \cong \frac{V_{m+2i-2}}{\im \, d_{m+2i-1}}.$$

Note $\C^{m+2}$   consists  of pairs $(\xi_{m+1},\lambda_{m}) \in \aut(V_{m+1})\times \aut(V_{m})$ with $d_{m+1} \circ \xi_{m+1} = \lambda_m \circ d_{m+1}$  
Since $(G_1)_R \cong V_{m}/  \mathrm{Im}\,d_{m+1}$ we deduce that  $\C^{m+2} \cong \aut(G_R)$. Since  $H_{m+1}(L(V_{\leq m}))=0$, invoking Theorem \ref{t1} we deduce $ \E(L(V_{\leq m+1})) \cong \aut(G_R).$

Now proceed by induction. Assume $r > 1$ and $r \leq m/2,$ as hypothesized.  The latter assumption ensures $ H_{m+2r-1}(L(V_{\leq m+2r-2}))=  H_{m+2r-2}(L(V_{\leq m+2r-3}))=0.
 $
Since $d_{m+2r-2} = 0$ we have  $$\C^{m+2r-1} = \aut(V_{m+2r-2})\times \E(L(V_{\leq m+2r -3}, \partial_{\leq m+2r -3})).$$   Since $H_{m+2r-2}(L(V_{\leq m+2r-3})) =0$, by Theorem \ref{t1} and the induction hypothesis  we obtain $$\E(L(V_{\leq m+2r-2}), \partial_{\leq m+2r -2}) \cong \aut(V_{m+2r-2})\times \prod_{i=1}^{r-1} \aut((G_i)_R),$$

 Finally, note  $\C^{m+2r}$ is the set  of triples  $$(\xi_{m+2r-1}, \lambda_{m+2r-2}, \alpha) \in  \aut(V_{m+2r-1})\times \aut(V_{m+2r-2})\times \prod_{i=1}^{r-1} \aut((G_i)_R)$$ such that the pair $(\xi_{m+2r-1}, \lambda_{m+2r-2}) \in \aut(V_{m+2r-1})\times \aut(V_{m+2r-2})$ gives a commutative diagram:
 $$
\xymatrix{  V_{m+2r-1} \ar[d]^{d_{m+2r-1}}  \ar[rr]^{\xi_{m+2r-1}} && V_{m+2r-1} \ar[d]^{d_{m+2r-1}}   \\
V_{m+2r-2} \ar[rr]^{\lambda_{m+2r-2}} && V_{m+2r-2}   
}
$$
 We conclude $\C^{m+2r} \cong  \aut((G_r)_R) \times \prod_{i=1}^{r-1} \aut((G_i)_R)$.   Theorem \ref{t1} and the fact that $H_{m+2r-1}(L(V_{\leq m+2r-2})) =0$   now completes the induction and the proof. \end{proof}

 We deduce the following  direct consequence of Theorem  \ref{main} (2).
\begin{corollary} \label{main3}  Let $R \subseteq \Q$ and 
 $X = X^q = X^n \cup_\alpha \left( \bigcup_{i=1}^j e_i^q \right)$ satisfy Hypothesis \ref{eq:hyp}.   When $q= n+1$ suppose further that the linking homomorphism  $\pi_{n+1}(X^{n+1}, X^{n})_R\to \pi_{n}(X^{n}, X^{n-1})_R$ vanishes. 
Then  $\mathcal{E}_{*}(X^{n})_R = 0$ implies $\mathcal{E}_{*}(X^{q})_R \cong \bigoplus_{i=1}^{j} \pi_{q}(X^{n})_R. \qed$ 
\end{corollary}

We apply this result to give a calculation of  $\E_*(X)_R$.   
Given $R$-modules $H_{i}, H_{j}$, let   $[H_i, H_j] \subseteq(H_{i}\otimes H_{j})\oplus (H_{j}\otimes H_{i})$ denote the $R$-submodule given by:
 $$[H_{i},H_{j}] = R\langle x_{i}\otimes x_{j}-(-1)^{(i-1)(j-1)}x_{j}\otimes x_{i} \mid  x_{i}\in H_{i},  x_{j}\in H_{j} \rangle$$ 
 Here we write $R \langle \, \, \rangle$ to denote a free $R$-module on   given generators.   
 Given subspaces $V, W \subseteq L$ with $L$ a graded Lie algebra we similarly write $$[V, W] = R\langle [v, w] \mid v \in V, w \in W \rangle.$$ 
We prove
\begin{theorem}\label{c4}
   Let $R \subseteq \Q$ and $X$ in $\textbf{CW}_{m}^{k+1}$ satisfy Hypothesis \ref{eq:hyp}. Suppose $X$ has a cellular decomposition of the form $$X^{m+1}\subset X^{m+2}\subset\cdots\subset X^{2m}=X$$ such that  that the linking homomorphisms $\pi_{r+1}(X^{r+1}, X^{r})_R\to \pi_{r}(X^{r}, X^{r-1})_R$ vanish for $r = m+2, \ldots, 2m-1$.  Suppose  $\mathcal{E}_{*}(X^{m+1})_R = 0$.   Then 
$$\mathcal{E}_{*}(X^{i})_R = 0, \hbox{\, for \,} i\leq 2m-1 \hbox{\, and \, } \mathcal{E}_{*}(X^{2m})_R  \cong \bigoplus_{i=1}^{j} [H_{m+1}(X; R), H_{m+1}(X;R)], $$
where  $j = \dim_R(H_{2m}(X, X_{2m-1}; R))$. 
\end{theorem}
\begin{proof}
By the freeness of the Anick model as DG Lie algebra over $R$ we obtain$$
H_{m+1}(L(V_{\leq m})) = H_{m+2}(L(V_{\leq m+1})) = \cdots = H_{2m-1}(L(V_{\leq 2m-2}))=0.$$   Since, by hypothesis,  $\mathcal{E}_{*}(X^{m+1})_R = 0$, applying Corollary \ref{main3} repeatedly gives
$$\mathcal{E}_{*}(X^{m+2})_R  = \mathcal{E}_{*}(X^{m+3})_R = \cdots = \mathcal{E}_{*}(X^{2m-1})_R = 0.$$
Applying this result again then gives
$$\mathcal{E}_{*}(X^{2m})_R \cong  \bigoplus_{i=1}^{j}   \pi_{2m+1}(X)_R.$$
Using the Anick model, we compute  $$ \pi_{2m+1}(X)_R \cong [V_m, V_m]
\cong [H_{m+1}(X; R), H_{m+1}(X;R)].$$
\end{proof}
The following result was stated as  Example \ref{exmain2}  in the introduction.
\begin{corollary} Let $G$ be a finitely generated abelian group and $R \subseteq \Q$ with least invertible prime $p$. 
Let $X$ be in $\textbf{CW}_{m}^{k+1}$ as in Hypothesis \ref{eq:hyp} with cellular decomposition of the form
$$M(G,m+1)=X^{m+2}\subset X^{m+3}\subset\cdots\subset X^{2m}=X.$$
Assume that the linking homomorphisms $\pi_{r+1}(X^{r+1}, X^{r})_R\to \pi_{r}(X^{r}, X^{r-1})_R$ vanish for $r = m+2, \ldots, 2m-1$ and that 
$\dim_R(H_{2m}(X, X^{2m-1};  R)) = j$. 
  Then
$$ \E_*(X)_R \cong \bigoplus_{i=1}^{j} [G_R, G_R]$$
where $G_R = H_{m+1}(M(G, m+1); R)$ is of degree $m+1$.  \end{corollary}
\begin{proof} By \cite[Theorem 3.2]{A-M},  $\E_{*}(M(G,m+1)) \cong \Z_{2}\oplus\Z_{2}\oplus \cdots \oplus \Z_{2}$. By Hypothesis \ref{eq:hyp},  the prime $p =2$ is invertible in $R$.  Thus $\E_{*}(M(G,m+1))_{R} =0$. The result now follows from  Theorem \ref{c4}.
\end{proof}
We can easily  compute both groups for a tame product of spheres. The following was stated as Example \ref{exmain3} in the introduction. 
\begin{theorem} Let $R \subseteq \Q$ be a ring with least invertible prime $p$  Let $m \leq n$ be chosen so that $m, k = m+n+1$ and $ p$ satisfy Hypothesis \ref{eq:hyp}.   Then
$$\begin{array}{ll} \E \left( (S^{m+1} \times S^{n+1})_R \right) & \cong \left\{  \begin{array}{ll}  R^*  \times  R^*  & \hbox{for }  n \neq 2m \hbox{ or } m \hbox{ even.} \\
 R \oplus (R^*  \times  R^*) &  \hbox{for } n = 2m \hbox{ and } m \hbox{ odd }
\end{array} \right.
\\ \\
 \E_* (S^{m+1} \times S^{n+1})_R  & \cong \left\{  \begin{array}{ll}   0 & \hbox{for }  n \neq 2m \hbox{ or } m \hbox{ even.}  \\
R  & \hbox{for } n = 2m \hbox{ and } m \hbox{ odd }.
\end{array} \right.
\end{array}
$$  
\end{theorem}
\begin{proof}
Let $X = S^{m+1} \times S^{n+1}$.  We can write the Anick model as $(L(u, v, w), \partial)$ where $|u| = m, |v| = n, |w| = n+m+1$ with $\partial(u) = \partial(v) = 0$ and $\partial(w) = [u, v].$   Theorem \ref{t1}   gives a  short exact sequence $$ 0 \to H_{m+n+1}(L(u, v)) \to  \E(L(u,v,w), \partial) \to \C^{m+n+2} \to 0.$$
For degree reasons, $H_{m+n+1}(L(u, v)) = 0.$  Thus $\E(X_R) \cong \C^{m+n+2}$ and we 
compute the latter group.

The group $\C^{m+n+2}$ consists of pairs $(\xi, [\alpha]) \in \aut(R\langle w \rangle) \times\E(L(u, v), 0)$ such that the following diagram commutes:
$$\xymatrix{R\langle w \rangle \ar[d]_{B_{m+n+1}} \ar[rr]^\xi &&  R\langle w \rangle \ar[d]^{B_{m+n+1}} \\
R\langle[u, v] \rangle \ar[rr]^{H_{m+n}(\alpha)} && R\langle [u, v] \rangle}
$$
The map $\xi$ is determined by $H_{m+n}(\alpha).$ Thus $\C^{m+n+2} \cong \E(L(u, v), 0).$
Applying Theorem \ref{t1} again gives a   short exact sequence:
$$ 0 \to  H_{n}(L(u)) \to \E(L(u, v), 0) \to  \C^{n+1} \to 0.$$
The sequence splits since the differential   vanishes.  As above,  $\C^{n+1} \cong  \aut(L(u, v)) \cong R^* \times R^*.$ 
If $n \neq 2m$ or $m $ odd then $H_n(L(u)) = 0.$    Otherwise, for $n = 2m$ and $m$ even,   $H_{n}(L(u)) = R$.  The result for $\E(X_R)$ follows.

The proof for $\E_*(X)_R$ is similar.  Theorem \ref{t5} gives  $\E_*(L(u, v, w), \partial) \cong \C_*^{m+n+2}$. As above, $\C_*^{m+n+2} \cong \E_*(L(u, v); 0).$  Applying Theorem \ref{t5} again gives  $\E_*(L(u, v); 0) \cong  H_n(L(u))$ and the result follows.  \end{proof}

We next give a  general result showing that finiteness of $\E(X_R)$ is not preserved by cell attachments in consecutive degrees: 
\begin{theorem}
\label{finite}  Let $R \subseteq \Q$ and and $X$ in  $\textbf{CW}_{m}^{k+1}$ be as in Theorem \ref{main}.  Suppose $X = X^n \cup (\bigcup_{i=1}^{j} e^{n+1})$ for $j > 0.$   Assume  that the linking homomorphism   $\pi_{n+1}(X^{n+1}, X^{n})_R \to \pi_n(X^n, X^{n-1})_R$ vanishes.    Then  $\E(X^n_R)$   finite implies  $\E(X^{n+1}_R)$ is infinite.
\end{theorem}
\begin{proof}
   Our hypothesis on the linking homomorphism ensures $d_{n} \colon V_{n} \to V_{n-1}$ is zero.
   Thus given $v \in V_n$, $\partial(v) \in L^{n-1}(V_{< n-1}).$   Since $\E(X^n_R) \cong \mathcal{E}(L( V_{<
   n}), \partial_{< n})$ is finite, applying  Theorem \ref{t1} gives $H_{n-1}(L(V_{< n-1}))=0$.   It follows that  the map $B_n = 0 \colon V_n \to H_{n-1}(L(V_{<n})).$   Given $v \in V_{n}$, set $\gamma= \mathrm{id} \colon (L( V_{<
   n}),\partial)\to (L( V_{<
   n}),\partial)$  and   $\xi^{a}\in \mathrm{aut}(V_{n}), a\in R$ with
$\xi^{a}(v)=av$ for $v\in V_{n}.$
The following diagram is obviously commutative:

$$\xymatrix{V_{n} \ar[d]_{B_n = 0} \ar[rr]^{\xi^a} && V_{n} \ar[d]^{B_n=0} \\
H_{n-1}(L(V_{<n})) \ar[rr]^{H(\alpha) = \mathrm{id}}  && H_{n-1}(L(V_{<n})).} $$ 
Therefore there exists an infinity of pairs
$(\xi^{a},[\mathrm{id}])\in \mathcal{C}_{n+1}$. Since   $\mathcal{C}_{n+1}$  is  infinite, $\E(X^{n+1}_R) \cong \mathcal{E}(L( V_{\leq n},\partial_{\leq n})$  is infinite by Theorem \ref{t1}.
\end{proof}
When $R= \Q$ the result becomes:
\begin{corollary}
\label{c1} Let $X$ be a finite, simply connected  CW complex. Suppose that $\pi_{n+1}(X^{n+1}, X^n)_\Q \neq 0.$ Then $\E(X^n_\Q)$ finite implies $\E(X_\Q^{n+1})$ is infinite.   \qed
\end{corollary}

 Finally, we show $\E(X_R)$ is   infinite for CW complexes of small dimension.

 \begin{theorem} \label{thm:4}  Let $X$ be a simply-connected finite CW complex  of dimension $\leq 5.$  Let $ R \subseteq \Q$ have least invertible prime $p \geq 7.$  If $p$ is finite, assume the linking homomorphisms $\pi_{r+1}(X^{r+1}, X^{r})_R\to \pi_{r}(X^{r}, X^{r-1})_R$ vanish for $r =  2,3,  4$.     Then $\E(X_R)$ is infinite. \end{theorem}
\begin{proof}  Let $(L(V), \partial)$ denote the Anick  model for $X$.  Our hypothesis on the linking homomorphism implies $d = 0.$  Since $X_R$ is not contractible, $V \neq 0.$ 

When  $X = X^3$ then for degree reasons $\partial = 0.$    It follows that $\E(X^3_R) \cong \aut(V)$ which is infinite since $V \neq 0.$ 

Next suppose $X = X^4.$  By Theorem \ref{main} (1), 
 it suffices to show that $\C^4$ is infinite.  Here  $V = V_3 \oplus V_2 \oplus V_1$ with $\partial(V_3) \subseteq [V_1, V_1]$ and at least one $V_i \neq 0.$   The group $\C^4$ consists of pairs $(\xi, \alpha)$ where $\alpha \colon L(V_2 \oplus V_1) \to L(V_2 \oplus V_1)$ is an automorphism and $\xi \colon V_3 \to V_3$ makes the diagram commute: 
$$ \xymatrix{V_3 \ar[rr]^{\xi} \ar[d]_{\partial_3} && V_3 \ar[d]^{\partial_3} \\
[V_1, V_1] \ar[rr]^{\alpha} && [V_1, V_1]}
$$
Given $a \neq 0$, define $\alpha^a(v) = av $ for $V_{\leq 2}$ 
Define $\xi^{a^2}$ by $\xi^{a^2}(v) = a^2v$ for $v \in V_3.$  
This gives  an infinity of distinct pairs $(\alpha^a, \xi^{a^2})$ in $\C^{4}.$

For the case $X = X^5$  we have  $V = V_{\leq 4}$.    
Again   $\partial_{3}(V_{3}) \subseteq [V_{1},V_{1}]$ by minimality. 
We identify $H_4(L(V_{\leq 3}))$ as vector space: 
$$ H_{4}(L(V_{\leq 3})) =  [\ker \, \partial_{3}, V_1]\oplus
[V_{2},V_{2}]\oplus \frac{[V_{2},[V_{1},V_{1}]]}{ [\partial_{3}(V_3), V_2]}
\oplus\frac{[V_{1},[V_{1},[V_{1},V_{1}]]]}{[\partial_3(V_3), [V_1, V_1]]}. $$
By Theorem \ref{t1},  if $H_4(L(V_{\leq 3})) \neq 0$ then $\E(X^5_\Q)$ is infinite.
Thus we may assume $H_4(L(V_{\leq 3})) = 0$ which forces 
$\dim V_2 \leq 1.$  If $V_1 =0$ then $\partial =0$ and $\E(X_R) \cong \aut(V)$ is infinite.  So assume $V_1 \neq 0.$ Then we must have $\ker \, \partial_3 = 0$ and $\partial_3(V_3) = [V_1, V_1]$.
Again, it suffices to show $\C^5$ is infinite.  We note that 
$$H_{3}(L(V_{\leq 3})) =  \ker\, \partial_3 \oplus 
[V_{2},V_{1}]\oplus \frac{[V_{1},[V_{1},V_{1}]]}{  
[ \partial_{3}(V_3), V_{1}]} \cong  [V_2, V_1]
$$
 Then $\C^5$ is the set of pairs $(\xi, [\alpha])$ with $\xi \in \aut(\Q \langle w \rangle)$ and $\alpha \in \aut(L(V_{\leq 3}), \partial_{\leq 3})$ making the diagram commute:
  $$ \xymatrix{  V_4 \ar[rr]^{\xi} \ar[d]_{\partial_4} && V_4 \ar[d]^{\partial_4} \\
  [V_2, V_1]  \ar[rr]^{H(\alpha)} && [V_2, V_1].}$$
  Given $a \in R^*$ define a DG Lie map $\alpha^a \colon (L(V_{\leq 3}), \partial_{\leq 3}) \to  (L(V_{\leq 3}), \partial_{\leq 3})$ by setting 
  $$\alpha(v) = av \hbox{ for } v \in V_1 \hbox{ and }   \ \ \alpha(u) = a^2u \hbox{ for } u \in V_3 \oplus V_2,$$
  and extending.   We then obtain an infinity of pairs $(\xi^{a^3}, [\alpha^a]) \in \C^5$ where $\alpha^{a^3}(x) = a^3x$  for $x \in V_4.$ 
  \end{proof}
We conclude by proposing a problem. By Costoya-Viruel \cite{CV}, every finite group $G$ occurs as $\E(X_\Q)$.  
The construction of $X$ for a given $G$ requires cells (cohomology classes) in a wide range of dimensions.  This suggests the following: 
\begin{problem}  Given $R \subseteq \Q$, find the smallest $n \geq 1$ such that there exists a simply connected CW complex $X$ with $\dim X \leq n$,  $X_R$ non-contractible and $\E(X_R)$ is finite.
\end{problem}
 \bibliographystyle{amsplain}

\begin{thebibliography}{10}

\bibitem{An1} D. J. Anick, {\em Hopf algebras up to homotopy,}
J. Amer. Math. Soc. {\bf 2} (1989) no. 3 417-453. 

\bibitem{An2} D. J. Anick,
{\em R-local homotopy theory},
Lecture Notes in Math., {\bf 1418}, (1990) 78-85.


\bibitem{AL1} M. Arkowitz and  G. Lupton, {\em On the nilpotency of subgroups of self-equivalences,} Progr. Math. {\bf 136} (1996), 1-22.
\bibitem{AL} M. Arkowitz and  G. Lupton, 
{\em Rational obstruction theory and rational homotopy sets,}
Math. Z. {\bf 235} (2000), no. 3, 525-539. 

\bibitem{A-M}  M. Arkowitz and K. Maruyama, {\em Self-homotopy equivalences which induce the identity on homology, cohomology or homotopy groups,} Topology Appl. {\bf 87} (1998), no. 2, 133-154.

\bibitem{A-S} M. Arkowitz and J. Strom,
{\em The group of homotopy equivalences of products of spheres and of Lie groups}, 
Math. Z. {\bf 240} (2002), no. 4, 689-710
\bibitem{Ba} H. J. Baues, {\em Homotopy type and homology}, Oxford Mathematical Monographs,  Oxford University Press, New York,  1996. 

\bibitem{B-B} W.   Barcus and  M. Barratt,  
{\em On the homotopy classification of the extensions of a fixed map,}
Trans. Amer. Math. Soc. {\bf 88} (1958),  57-74. 


%
%
\bibitem{Benk} M. Benkhalifa, 
{\em Realizability of the group of rational self-homotopy equivalences,}  
J. Homotopy Relat. Struct. {\bf 5} (2010), no. 1, 361-372. 
\bibitem{CV} C. Costoya and A. Viruel, \emph{Every finite group is the group of self homotopy equivalences of an elliptic space,} preprint;
 arXiv:1106.1087


\bibitem{DZ}  E. Dror and  A. Zabrodsky,
{\em Unipotency and nilpotency in homotopy equivalences,}
Topology {\bf 18} (1979), no. 3, 187-197. 
\bibitem{FF} J. Federinov and Y. F\'{e}lix, {\em Realization of 2-solvable nilpotent groups as groups of classes of homotopy self-equivalences,} Topology Appl. {\bf 154}, no. 12,  (2007), 2425-2433 
\bibitem{Kahn} D. W. Kahn, {\em Some research problems on homotopy-self-equivalences,}   Springer Lecture Notes in Math., {\bf 1425}  (1990),  204-207
\bibitem{Mar}  K. Maruyama
{\em Localization of a certain subgroup of self-homotopy equivalences,}
Pacific J. Math. {\bf 136} (1989), no. 2, 293-301. 
\bibitem{Q} D. Quillen, {\em Rational homotopy theory,} 
Ann. Math. (2) {\bf 90} (1969), 205-295.
\bibitem{Ru} J. Rutter, {\em Spaces of homotopy self-equivalences: A survey,} Lecture Notes in Mathematics, {\bf 162},  Springer-Verlag, Berlin (1997). 
\bibitem{Tanre} D. Tanr\'{e}, \textit{Homotopie rationnelle; mod\`eles de Chen, Quillen, Sullivan}, 
Lecture Notes in Mathematics  1025, Springer, Berlin, 1983.
\end{thebibliography}

\end{document}